\documentclass[11pt,oneside]{amsart}

\usepackage{hyperref,graphicx}
\usepackage[scale=0.7]{geometry}
\usepackage{color}

\newcommand{\abs}[1]{\lvert#1\rvert}

\newcommand{\F}{\mathcal{F}}

\newcommand{\Z}{\mathbb{Z}}
\newcommand{\R}{\mathbb{R}}
\newcommand{\N}{\mathbb{N}}
\newcommand{\Q}{\mathbb{Q}}

\newcommand{\real}{\mathbb{R}}

\newcommand{\bigsetcond}[2]{\bigl\{ #1 \,:\, #2 \bigr\}}

\newcommand{\vi}{{\sf V}}

\newtheorem{theorem}{Theorem}[section]

\newtheorem{lemma}[theorem]{Lemma}
\newtheorem{coro}[theorem]{Corollary}

\theoremstyle{definition}
\newtheorem{definition}[theorem]{Definition}
\newtheorem{example}[theorem]{Example}
\theoremstyle{remark}

\numberwithin{equation}{section}

\begin{document}
\title[Chance-Constrained $S$-Convex Optimization]{Beyond Chance-Constrained Convex Mixed-Integer Optimization: {\small A Generalized Calafiore-Campi Algorithm and the notion of $S$-optimization.}}

\author{J. A. De Loera}
\author{R. N. La Haye}
\author{D. Oliveros}
\author{E. Rold\'an-Pensado}
\address[J. A. De Loera, R. N. La Haye]{Department of Mathematics, UC Davis}
\email{deloera@math.ucdavis.edu, rlahaye@math.ucdavis.edu}
\address[D. Oliveros, E. Rold\'an-Pensado]{Instituto de Matem\'aticas, UNAM campus Juriquilla}
\email{dolivero@matem.unam.mx, e.roldan@im.unam.mx}


\begin{abstract}
The scenario approach developed by Calafiore and Campi to attack chance-constrained convex programs (i.e., optimization problems 
with convex constraints that are parametrized by an uncertainty parameter) utilizes random sampling on the uncertainty parameter to substitute the original problem with a representative  continuous convex optimization with $N$ convex constraints which is a relaxation of the original. Calafiore and Campi provided an explicit estimate 
on the size $N$ of the sampling relaxation to yield high-likelihood feasible solutions of the chance-constrained problem. They  measured the probability  of the original constraints to be violated by the random optimal solution from the relaxation of size $N$.

This paper has two main contributions.  First, we present a generalization of the Calafiore-Campi results to both integer and mixed-integer variables. 
In fact, we demonstrate that their sampling estimates work naturally for variables that take on even more sophisticated values
restricted to some subset $S$ of $\R^d$. In this way, a sampling or scenario algorithm for chance-constrained convex mixed integer 
optimization algorithm is just a very special case of a stronger sampling result in convex analysis. The key elements, necessary for all the proofs, 
are generalizations of Helly's theorem where the  convex sets are required to intersect  $S \subset \R^d$. 
The size of samples in both algorithms will be directly determined by the $S$-Helly numbers.  

Motivated by the first half of the paper, for any subset $S \subset \R^d$, we introduce the notion of  an $S$-optimization problem, 
where the variables take on values over $S$. It generalizes continuous ($S=\R^d$), integer ($S=\Z^d$), and mixed-integer 
optimization ($S=\R^k \times \Z^{d-k}$). We illustrate with examples the expressive power of $S$-optimization to capture sophisticated 
combinatorial optimization problems with difficult modular constraints. We reinforce the evidence that $S$-optimization is ``the right concept'' 
by showing that the well-known randomized sampling algorithm of K. Clarkson for low-dimensional convex optimization problems can be extended 
to work with variables taking values over $S$.

\end{abstract}

\maketitle

\section{Introduction}

Chance-constrained optimization is a branch of stochastic optimization concerning problems in which constraints are imprecisely known but the problems need to be solved with a minimum probability of reliability or certainty.  Such problems arise quite naturally in many areas of finance (e.g., portfolio planning where losses should not exceed some risk threshold) \cite{gaivoronski2005value,pagnoncelli2008computational}, telecommunications (services agreements where contracts require network providers to guarantee with high probability that packet losses will not exceed a certain percentage) \cite{marianov2000probabilistic,venetsanopoulos1986topological}, and facility location (for medical emergency response stations, while requiring high probability of coverage over all possible emergency scenarios) \cite{aly1978probabilistic,beraldi2004designing}.  

Chance-constrained problems are notoriously difficult to solve because the feasible region is often not convex and because the probabilities can be hard to compute exactly. 
For information on how to solve such chance-constrainted problems and how to deal with probabilistic uncertain optimization see \cite{calafiorecampi2005,campigaratti, shapirodentchevarusz,prekopaetal, luedtke2008sample,pagnonetal,vielmaetal} and the excellent references therein.

\noindent We have two main contributions:


\subsection*{Sampling in Chance-constrained Convex Mixed Integer Optimization and beyond}
Our main result is a generalization of the \emph{scenario approximation method} of Calafiore and Campi
\cite{calafiorecampi2005,calafiorecampi2006} for continuous variables and convex constraints. 
Here we generalize their sampling algorithm for integer and mixed-integer variables. To state the result we need the following notions.
Let $\Omega$ be a probability space. Let  $f(\mathbf x,w)\colon ( \Z^{d-k}\times\R^k) \times \Omega\to\R$  be a convex function on $\mathbf x \in  \Z^{d-k}\times\R^k$ and measurable on $w \in \Omega$.
This function $f$ can be thought of as representing  constraints on $ \Z^{d-k}\times\R^k$, one for each value of $w$. 
Note that  ``$\mathbf x$ violates the constraint'' is a random event. The \emph{probability of violation} of a vector $\mathbf x\in \Z^{d-k}\times\R^k$ is defined as
$V(\mathbf x) = Pr[\{w\in\Omega:f(\mathbf x,w)>0\}].$ We seek a solution $\mathbf x$ with small associated value for $V (x)$, because it means it is feasible for  ``most'' of the problem instances. We also hope for our conclusion to hold with high confidence, or equivalently we wish to have small amount of distrust for the prediction.

\begin{coro} \label{gen-calafiorecampimip} 
Let $f$ and $\Omega$ be given as above. Let $0< \epsilon \leq 1$ (tolerance for violation), $0<\delta<1$ (distrust or lack of confidence) be chosen parameters. 
Suppose further that there is an  optimal value $\mathbf x_*$ of the linear minimization chance-constrained mixed-integer convex problem 

\begin{equation*}
\begin{split}
    \min \quad & c^T \mathbf x \\
   \text{subject to} \quad & V(\mathbf x)\leq \epsilon \\
	& {\mathbf x} \in K \, \text{convex set}, \\
    &{\mathbf x} \in \Z^{d-k}\times\R^k.
\end{split}
\end{equation*}

Then from a sufficiently large-size random sample of  $N$ different $i.i.d.$ values for $w$ (specifically, $w^1, w^2,\dots,w^N$), $\mathbf x_*$ 
can be $\delta$-approximated by the random variable $\mathbf x_N$, the optimal solution of the convex mixed-integer optimization problem 

\begin{align*} 
    \min \quad & c^T \mathbf x \\
    \text{subject to} \quad & f(\mathbf x, w^i) \leq 0 ,\quad i=1,2,\ldots, N, \\
	& \mathbf x \in K \text{ convex set}, \\
	& \mathbf x \in \Z^{d-k}\times\R^k.
\end{align*}

More precisely, if $x_N$ exists and the sample has size $N \geq \frac{2(2^{d-k}(k+1)-1)}{\epsilon} \ln(1/\epsilon) + \frac{2}{\epsilon} \ln(1/\delta)+2(2^{d-k}(k+1)-1),$ then
the undesirable event of high-infeasibility $V(x_N)> \epsilon$ has probability less than $\delta$ of occurring. 
\end{coro}

Note that when $k=0$, we are in the situation of chance-constrained integer convex optimization, which is a special case. In fact,
Corollary \ref{gen-calafiorecampimip} follows from a more general result. But before we can state it, we need one important definition on convex analysis.

\begin{definition}
For a nonempty family $\mathcal K$ of sets, the \emph{Helly number} $h=h(\mathcal K)\in\N$ of $\mathcal K$ is defined as the smallest number satisfying the following:
\begin{align*}\label{helly:cond}
	& \forall i_1,\ldots,i_h \in [m] : F_{i_1} \cap \cdots \cap F_{i_h} \neq \emptyset && \Longrightarrow && F_1 \cap \cdots \cap F_m \neq \emptyset
\end{align*}
for all $m \in\N$ and $F_1, \ldots, F_m \in \mathcal K$. If no such $h$ exists, then $h(\mathcal K):=\infty$. 

E.g., for the classical Helly's theorem that appears in all books in convexity, $\mathcal K$ is the family of all convex subsets of $\R^d$.

For $S \subseteq \real^d$ we define
\begin{equation*} 
	h(S) := h\bigl( \bigsetcond{S \cap K}{K \subset \real^d \ \text{is convex\,} }\bigr).
\end{equation*}
That is, $h(S)$ is the Helly number when the sets are required to intersect at points in $S$; we will call this the \emph{$S$-Helly number}. 
\end{definition}

For instance, when $S$ is finite then the bound $h(S)\le |S|$ is trivial. The original Helly number is $h(\R^d)=d+1$ and, interestingly, if $\F$ is any subfield of $\R$ (e.g., $\Q(\sqrt{2})$), then Radon's proof of Helly's theorem directly shows that the $S$-Helly number of $S=\F^d$ is still $d+1$. Doignon' theorem \cite{Doi1973} (later rediscovered in \cite{Bel1976,hoffman,Sca1977}) states that a finite family of convex sets in $\R^d$ intersect at a point of $\Z^d$ if every $2^d$ of members of the family intersect at a point of $\Z^d$. Another example is the work of A.J. Hoffmann in \cite{hoffman} and Averkov and Weismantel \cite{AW2012}  who gave a \emph{mixed} version of Helly's and Doignon's theorems which includes them both. This time the intersection of the convex sets is required to be mixed-integer, with variables taking values in $\Z^{d-k}\times\R^k$, and this can be guaranteed if every $2^{d-k}(k+1)$ sets intersect in such a point. 

The $S$- Helly number $h(S)$ is relevant for our purposes as it is a measure of the feasibility of a system of convex constraints over $S$. Essentially, if the system is $S$-infeasible,
then there must be a subsystem of size $h(S)$ or less that is infeasible. E.g., from the classical Helly's theorem one derives that, given (real ) infeasible convex constraints in $d$ variables would contain a subset of no more than $d+1$ constraints that certifies that the entire set has empty intersection (no common solution).
It is fair to say that applications in optimization have prompted many papers about Helly numbers \cite{AW2012,Bel1976,clarkson,ViolatorSpaces2008,hoffman,Sca1977}. In \cite{calafiorecampi2005,calafiorecampi2006}  the usual Helly number $d+1$ played a role for the size of a support set, if we are interested on solutions with values on $S$ we can use the $S$-Helly number to predict the size of a support set and recover the sample size.


We state here our most general theorem (for full details see Section \ref{CCstuff}).

\begin{theorem} \label{gen-calafiorecampi} Let $S \subseteq \R^d$ be a set with a finite Helly number $h(S)$. 
Let $0< \epsilon \leq 1$ (tolerance), $0<\delta<1$ (distrust) be chosen parameters.
Let $f({\mathbf x},w)$ be a convex function in $\mathbf x$ and  measurable in $w$. Suppose there is an  
optimal value $\mathbf x_*$ of the linear minimization chance-constrained problem 

\begin{equation*}
\begin{split}
  CCP(\epsilon)=  \min \quad & c^T \mathbf x \\
     \text{subject to} \quad & Pr[f(\mathbf x, w) > 0] < \epsilon, \\
	& {\mathbf x} \in K \, \text{convex set}, \\
    &{\mathbf x} \in S.
\end{split}
\end{equation*}

Then from a sufficiently large random sample of  $N$ different $i.i.d$ values for $w$ (specifically, $w^1, w^2,\dots,w^N$), $\mathbf x_*$ 
can be $\delta$-approximated by $\mathbf x_N$,  the optimal solution of the convex optimization problem 

\begin{align*}
  SCP(N)=  \min \quad & c^T \mathbf x \\
    \text{subject to} \quad & f(\mathbf x, w^i) \leq 0 ,\quad i=1,2,\ldots, N, \\
	& \mathbf x \in K \text{ convex set}, \\
	& \mathbf x \in S.
\end{align*}
More precisely, if $x_N$ exists and  the size of the sample $N \geq \frac{2(h(S)-1)}{\epsilon} \ln(1/\epsilon) + \frac{2}{\epsilon} \ln(1/\delta)+2(h(S)-1),$ then
the undesirable event of high-infeasibility $V(x_N)> \epsilon$ has probability less than $\delta$ of occurring. 
\end{theorem}

Indeed, taking $S= \Z^{d-k}\times\R^k$ and because its Helly number is $h(S)=2^{d-k} (k+1)$,  we obtain as an immediate consequence the result for chance-constrained convex mixed integer optimization stated in Corollary \ref{gen-calafiorecampimip}. Moreover, we can provide the following guarantee of the quality of the solution:

\begin{theorem} \label{quality}
Let $0<\epsilon\leq1$ (tolerance), $0<\delta<1$ (confidence), and $N$ sufficiently large (as in Theorem \ref{gen-calafiorecampi}; note that this depends on $S$).
Let $J^\epsilon$ be the optimal objective value of $CCP(\epsilon)$ and $J^N$ be the optimal objective value of $SCP(N)$  (note that $J^N$ is a random variable).

\begin{enumerate}
\item
Suppose $CCP(\epsilon)$ is feasible. Then with probability of at least $1-\delta$, if $SCP(N)$ is feasible, it holds that $J^\epsilon\leq J^N$.

\item
Define $\epsilon_1=1-(1-\delta)^{1/N}$. With probability at least $1-\delta$, we have $J^N\leq J^{\epsilon_1}$.

\end{enumerate}
\end{theorem}

%
%


\subsection*{$S$-optimization and Clarkson's Algorithm}

The essential arguments used in the Calafiore-Campi scenario method apply to more complicated variable values over $S$, well beyond the
reals or the integers. The proofs are also the same. This motivated us to introduce the notion of $S$-optimization, a natural  generalization of continuous, integer, and mixed-integer optimization:

\begin{definition} 
Given $S \subset \R^d$, the optimization problem with equations, inequalities and variables taking values on $S$,
\begin{align*}
    \max \quad              & f(\mathbf x)  \\
    \text{subject to} \quad & g_i(\mathbf x) \leq 0, \quad i=1,2,\ldots, n, \\
                    & h_j(\mathbf x) = 0,\quad j=1,2,\ldots, m,\\                  &{\mathbf x} \in S,
\end{align*}
will be called an {\em $S$-optimization problem.} 
\end{definition}

Clearly when $S=\R^d$ the $S$-optimization problem is the usual continuous optimization problem, 
 $S=\Z^d$ is just integer optimization, and $S=\Z^k \times \R^{d-k}$ is the case of mixed-integer optimization. 
When only linear constraints are present this is an {\em $S$-linear program.} When all constraints are convex we call
this an \emph{$S$-convex program}. This paper presents two algorithmic results about $S$-convex programs.

But, why study $S$-optimization? Or, rather, why do it for an \emph{unfamiliar} set $S$? As we show below 
$S$-optimization problems  have natural expressive power, sometimes using fewer or simpler constraints than
standard continuous or mixed integer optimization.  

Here are two more unusual examples of $S$-optimization problems. This time we  model succinctly with more sophisticated 
$S \subset \R^d$ (typically discrete sets). 

\begin{example}
Given a graph $G=(V,E)$, we reformulate the classic graph $K$-coloring query as the solvability of the following linear system of modular inequations:
For all $(i,j)$ in $E(G)$ consider the inequations $c_i \not \equiv c_j \mod K$. This is a system on $\abs{V}$ variables and it has a solution if and only if the graph is $K$-colorable.
Note that the set of points $\mathbf c=(c_1,\dots,c_{\abs{V}})$ with $c_i \equiv c_j \mod K$ is a lattice, which we call $L_{i,j}$. Therefore, solving our system of inequalities is equivalent to finding a $\mathbf c\in S=\Z^{\abs{V}}\setminus(\bigcup_{i,j} L_{i,j})$. Consequently, the problem of deciding $k$-colorability is equivalent to the problem of finding a solution to an $S$-linear system of equations, where the variables take values on $S$,  the set difference of a lattice and a union of several sublattices.
\end{example}

\begin{example} Here is another instance of $S$-optimization which has ancestors in \cite{glover}.
We are interested in the the solutions of the following modular mixed-integer optimization problem:

\begin{align*}
	\min \quad & 3x_1+7x_2 +4x_3 +  \sum_{i \geq 4} ^{N}  (100-i) x_i \\
    \text{subject to} \quad & 8x_1+3x_2 + 5x_3 \equiv 6  \mod 11, \\
    & 6x_1+4x_2 -3x_3 \equiv 1 \mod 2, \\
	& x_1\not \equiv x_3 \mod 5,  \\
	& x_1+x_2+x_3+ \sum_{i \geq 6} ^{N}  x_i \leq 1000, \\
    &  x_1 \not\equiv 2,4,16 \mod 23,\quad x_2 \equiv 0 \mod 2,\quad x_3 \equiv 2 \mod 3,  \\
    & x_1,x_2,x_3 \geq 0 \text{ and integral,} \\
    & x_i, \, i=4,\dots,N \text{ continuous.}
\end{align*}

Note that only the integral variables have modular restrictions. 
By adding integer slacks, we can reformulate this problem as a  problem
with only six integral variables (with modular restrictions) and $N$ continuous variables.
\begin{align*}
	\min \quad & 3x_1+7x_2 +4x_3 +  \sum_{i \geq 4} ^{N}  (100-i) x_i \\
    \text{subject to} \quad & 8x_1+3x_2 + 5x_3=y_1 \\
    & 6x_1+4x_2 -3x_3=y_2, \\
	& x_1-x_3=y_3, \\
	& x_1+x_2+x_3+ \sum_{i \geq 6} ^{N}  x_i \leq 1000, \\
    & x_1 \not\equiv 2,4,16 \mod 23,\quad x_2 \equiv 0 \mod 2,\quad x_3 \equiv 2 \mod 3, \\
    & y_1 \equiv 6 \mod 11,\quad y_2 \equiv 1 \mod 2,\quad y_3 \not\equiv 0 \mod 5,  \\
	& x_1,x_2,x_3 \geq 0 \text{ and integral,} \\
	& x_i, \, i=4,\dots, N \text{ continuous.}
\end{align*}
  
What is the set $S \subset \R^6$ where the variables take on values for this situation? 
The answer can be described  first as $S_1 \times S_2 \times S_3 \times S_4 \times S_5 \times S_6 \times \R^{N-3}$, 
where $S_i$ can be described as the difference between $\Z$ and the subtraction of cosets (or translated sublattices) with respect to lattices of multiples
of an integer $q$. Thus at the the end $S_1 \times S_2 \times S_3 \times S_4 \times S_5 \times S_6$ can be written as the lattice $\Z^6$ from which we subtract the union of several translated sublattices.
\end{example}

We must remark that the Helly number for the difference of a lattice and a union of its sublattices has been estimated in \cite{justhelly}. 
To fully stress that $S$-optimization is the right notion we present a second algorithm which works well for $S$-convex  programs.

%
%
%
%
In \cite{clarkson},  K. Clarkson introduced a family of  algorithms which,  relying on repeated calls to an oracle that  optimizes small-size subsystems,  iteratively samples from the original (large) optimization problem until it reaches a global optimum. The expected runtime  is linear in the number of input constraints.  Our key observation is that  the same Clarkson ideas are applicable to $S$-convex optimization problems that have a  sampling size given by the Helly number $h(S)$ of the variable domain. Clarkson, in his ground-breaking work, applied his ideas already to traditional linear and integer linear optimization because he had the $S$-Helly numbers of $S=\R^d$ and $S=\Z^d$. 
However, he did not explicitly invoke the concept of the $S$-Helly number. By doing this, we now present a direct generalization of  his algorithms. Our proof that Clarkson's algorithm extends relies on the theory of violator spaces \cite{ViolatorSpaces2008}.

\begin{theorem}\label{ourclarkson}
Let $S \subseteq \R^d$ be a closed set with a finite Helly number $h(S)$. 
Using Clarkson's algorithm, one can find a solution of the $S$-convex optimization problem
\begin{align*} 
    \min \quad              & c^T \mathbf x \\
    \text{subject to} \quad & f_i(x)  \leq 0, \quad f_i \text{ convex for all }  i=1,2,\ldots, m,\\
               &{\mathbf x} \in S,
\end{align*}
in an expected $O\left(h(S)m+ h(S)^{O(h(S))}\right)$  calls to an oracle  that solves smaller subsystems of the system above of size $O(h(S))$. Thus, 
when a violation primitive oracle runs in polynomial time and $h(S)$ is small, Clarkson's algorithm runs in expected linear time in the number of constraints.
\end{theorem}


\section{ A Calafiori-Campi Style Algorithm for Chance-Constrained Convex $S$-optimization} \label{CCstuff}


We begin with some formal preliminaries. In all that follows, let $S$ be a proper subset of $\R^d$ and let $\Omega$ be a probability space.
Suppose we have a function $f(\mathbf x,w)\colon S\times \Omega\to\R$ which is convex on $\mathbf x \in S$ and measurable on $w$. (chance variables that represent stochasticity). 
This parametric function $f$ can be thought of as representing one constraint on $S$ for each value of $w$: given $\mathbf x\in S$, $\mathbf x$ satisfies the constraint if $f(\mathbf x,w)\leq 0$ and violates the constraint otherwise.  Note that  ``$\mathbf x$ violates the constraint'' is a random event.  We have the following definition:

\begin{definition}[\cite{calafiorecampi2005,calafiorecampi2006}]
Let $\mathbf x\in S$ be given. The \emph{probability of violation} of $\mathbf x$ is defined as
$$V(\mathbf x) = Pr[\{w\in\Omega:f(\mathbf x,w)>0\}].$$
\end{definition}

For example, if we take the uniform probability density  (with respect to Lebesgue's measure). 
Then $V (x)$ is just  the volume of those parameters $w$ for which $f(x,w) \leq 0$  is violated. 
We seek a solution $x$ with small associated value for $V (x)$, because it means is feasible for  ``most'' of the problem instances. 
When a vector $\mathbf x\in S$ has a small probability of violation $V(x)$, $\mathbf x$ is said to be \emph{approximately feasible} 
(this notion was first studied in \cite{barmish2000avoiding}).

Let $\epsilon \in [0,1]$ represent the \emph{tolerance for violation}. For $\bf x\in S$ if we have $ Pr[f(\mathbf x, w) \leq 0] \geq 1-\epsilon$, we say $x$ is an  $\epsilon$-level 
feasible solution.  In other words, $\epsilon$-level solutions are those with $V(x)<\epsilon$.  Moreover, we would like to be confident that high probability of violation is unlikely to
occur among the constraints, so we will use a \emph{distrust parameter} $\delta \in [0,1]$. When $\delta$ is small, it represents high confidence on the prediction.  Our goal is  to find
$x^*$ such that
$$Pr[\{V(x^*)\geq\epsilon\}] \leq \delta.$$

Given a linear cost function $\mathbf x\mapsto c^T\mathbf x$ and the tolerance $\epsilon$, 
a natural problem is that of minimizing $c^T\mathbf x$ over $\{\mathbf x\in S|V(\mathbf x)<\epsilon\}$. This is a chance-constrained $S$-convex optimization problem $CCP(\epsilon)$:

\begin{equation}  \label{Sconstrained}
\begin{split}
  CCP(\epsilon)=  \min \quad & c^T \mathbf x \\
    \text{subject to} \quad & V(\mathbf x) <\epsilon, \\
	& {\mathbf x} \in K \, \text{convex set}, \\
    &{\mathbf x} \in S.
\end{split}
\end{equation}

The key idea to solve $CCP(\epsilon)$ is to create a similar, but easier, problem. 
We may sample $N$ i.i.d. values $w^1, w^2,\dots,w^N$ from $\Omega$. 
This gives us the \emph{sampled convex program}

\begin{equation} \label{Nsample}
\begin{split}
  SCP(N)=  \min \quad & c^T \mathbf x \\
    \text{subject to} \quad & f(\mathbf x, w^i) \leq 0 ,\quad i=1,2,\ldots, N, \\
	& \mathbf x \in K \text{ convex set}, \\
	& \mathbf x \in S.
\end{split}
\end{equation}

Denote $\mathbf x_N$ the (uniquely selected) optimal solution of the problem \eqref{Nsample}. 
$SCP(N)$ can be solved to produce an optimal solution $\mathbf x_N$. Note that $x_N$ is a random variable.
Since $\mathbf x_N$ is random, $\{V(x_N)\geq\epsilon\}$ is an event with some probability.  In fact, 
$V(x^N)$ is a random variable in the space $\Delta^N$ with product probability measure $Pr\times Pr \times \dots \times Pr=Pr^N$.
We claim that  for large enough $N$, we have $Pr[\{V(x^N)\geq\epsilon\}] \leq \delta.$
  
The key point of the sampling algorithms is to show that this is satisfied for sufficiently large sample size $N$. For $S=\R^d$, Calafiore and Campi found what sufficiently large $N$ is necessary in \cite{calafiorecampi2005,calafiorecampi2006}. Here we extend the scenario approximation scheme from \cite{calafiorecampi2005,calafiorecampi2006} to the mixed-integer case, $S=\Z^k\times\R^{d-k}$. Note that this includes the pure integer case $S=\Z^d$ as well as other $S$. 

Before we start the proof of the main results we require a purely technical estimation:

\begin{lemma} If
\begin{equation} \label{larga}
N \geq \frac{1}{1-r}  \left(    \left(\frac{1}{\epsilon}\right)  \ln \left(\frac{1}{\delta}\right)  +   h  +  \left(\frac{h}{\epsilon}\right) \ln \left(\frac{1}{r \epsilon}\right) + \frac{1}{\epsilon} \ln \left( \left(\frac{h}{\varepsilon}\right)^h \left(\frac{1}{h!}\right) \right) \right),
\end{equation}
then $\binom{N}{h} (1-\epsilon)^{N-h} \leq \delta$.
\end{lemma}

\begin{proof}[\bf Proof]
Note that
$$N \geq \frac{1}{1-r}  \left(    \left(\frac{1}{\epsilon}\right)  \ln \left(\frac{1}{\delta}\right)    +    h  +  \left(\frac{h}{\epsilon}\right) \ln \left(\frac{1}{r \epsilon}\right) + \frac{1}{\epsilon} \ln \left( \left(\frac{h}{\epsilon}\right)^h \left(\frac{1}{h!}\right) \right)\right) $$
implies that

$$(1-r) N \geq   \left(\frac{1}{\epsilon}\right)  \ln \left(\frac{1}{\delta}\right)    +    h  +  \left(\frac{h}{\epsilon}\right) \left( \ln \left(\frac{h}{r \epsilon}\right) -1\right) - \frac{1}{\epsilon} \ln \left({h!}\right).$$
Thus
$$N \geq   \left(\frac{1}{\epsilon}\right)  \ln \left(\frac{1}{\delta}\right)    +    h  +  \left(\frac{h}{\epsilon}\right) \left( \ln \frac{h}{r \epsilon} -1+\frac{rN\epsilon}{h}\right) - \frac{1}{\epsilon} \ln ({h!}).$$

But then, using the fact that $\ln(x) \geq 1 - \frac{1}{x}$ for positive values of $x$ and applying it to $x=\frac{h}{rN\epsilon}$, we obtain 
$$ N \geq   \left(\frac{1}{\epsilon}\right)  \ln \left(\frac{1}{\delta}\right)    +    h  +  \left( \frac{h}{\epsilon}\right) \ln \left(N\right)  - \frac{1}{\epsilon} \ln ({h!}).$$

From this last equation one can bound the logarithm of $\delta$, in such a way that
$ \ln(\delta) \geq -\epsilon N + \epsilon h + h \ln(N) - \ln(h!)$. Therefore, using the fact that $e^{-\epsilon (N-h)} \geq (1-\epsilon)^{N-h}$ (because $-\epsilon \geq \ln(1-\epsilon)$), we obtain
$$ \delta \geq \frac{N^h}{h!} e^{-\epsilon (N- h)} \geq \frac{ N(N-1)\dots (N-h+1)}{h!} (1-\epsilon)^{N-h}.$$

This last inequality can be rewritten as $\delta \geq \binom{N}{h} (1-\epsilon)^{N-h}$, finishing the proof of the statement.
\end{proof}

We now present the proofs of Theorems \ref{gen-calafiorecampi} and \ref{quality}. Recall we are concerned with the linear minimization chance-constrained problem $CCP(\epsilon)$.
We assume that the problem has an optimal solution.

\begin{proof}[Proof of Theorem \ref{gen-calafiorecampi}] 
Suppose we have the sampling set $\{w^1,\dots,w^N\}$. Denote again by $x_N$ the optimal solution for the auxiliary problem \eqref{Nsample} obtained from the sampling. 
Note that because $f$ is convex, each choice $w^i$ gives  an $S$-convex set $K_i=\{ x \in K: f(x,w^i) \leq 0\}$. The proof will require the use of the $S$-Helly number $h(S)$. 
The most important fact to do the estimations is that if we have the optimum value of \eqref{Nsample}, the optimal solution is defined by no more than $(h(S)-1)$ of the $K_i$. 
This is because $K_i$, $i=1\dots N$, together with $c^Tx<c^Tx_N$, is a convex set which has no solutions in $S$. Thus, by the definition of the $S$-Helly number, there are no
more than $h(S)$ infeasible subfamilies; this means that from the original $K_i$ only $h(S)-1$ participate. We call these $h(S)-1$ subsets the \emph{witness constraints} of the problem \eqref{Nsample}.

Let $\Gamma_N$ be the set of all possible values  $N$ i.i.d samples can take $w^1,w^2,\dots,w^N$. Now consider all possible index sets  $I \subset [N]=\{1,\dots, N\}$ of cardinality $(h(S)-1)$ and define 
$$\Gamma_N^I=\left\{(w^1,\dots,w^N) \in \Gamma_N: (w^i)_{i \in I} \text{ defines the witness constraints of } \eqref{Nsample}\right\}.$$

Therefore, $\Gamma_N$ can be written as the union of the $\Gamma_N^I$ for all possible choices of $I$. Using this we will bound the probability that $x_N$ is not 
in the solution set of \eqref{Sconstrained}. For simplicity, let $$R_{\epsilon}=\{{\mathbf x} \in K \cap S:  \, Pr[  f(\mathbf x, w) \leq 0] \geq 1-\epsilon \} \quad \text{and} \quad G_\epsilon=(K \cap S) \setminus R_\epsilon.$$

\begin{align*}
Pr\left[ \left\{\left(w^1,\dots, w^N\right) \in \Gamma_N :  x_N \in G_{\epsilon} \right\}\right] \leq & \\
\sum_{I \subset \left[N\right], \abs{I}=\left(h\left(S\right)-1\right)}  Pr\left[ \left\{\left(w^1,\dots, w^N\right) \in \Gamma_N^I : x_I \in G_{\epsilon} \right\}\right] =& \\
\sum_{I \subset \left[N\right], \abs{I}=\left(h\left(S\right)-1\right)}  Pr\left[ \left\{\left(w^i\right)_{i\in I}:  x_I \in G_{\epsilon} \right\} \cap \left\{\left(w^i\right)_{i \in I}:  f\left(x_I,w^j\right) \leq 0, j \notin I\right\} \right] =& \\
\begin{aligned}
\sum_{I \subset \left[N\right], \abs{I}=\left(h\left(S\right)-1\right)} & Pr\left[ \left\{\left(w^i\right)_{i\in I}:  x_I \in G_{\epsilon} \right\}\right] \times \\ & Pr\left[ \left\{\left(w^i\right)_{i \in I}: f\left(x_I,w^j\right) \leq 0, j  \notin I\right\} \Bigm| \left\{\left(w^i\right)_{i\in I}:  x_I \in G_{\epsilon} \right\} \right]=
\end{aligned} & \\
\begin{aligned}
\sum_{I \subset \left[N\right], \abs{I}=\left(h\left(S\right)-1\right)} & Pr\left[ \left\{\left(w^i\right)_{i\in I}:  x_I \in G_{\epsilon} \right\}\right] \times \\
& \prod_{j \notin I} Pr\left[ \left\{\left(w^i\right)_{i \in I}:  f\left(x_I,w^j\right) \leq 0\right\} \Bigm| \left\{\left(w^i\right)_{i\in I}:  x_I \in G_{\epsilon} \right\} \right] \leq
\end{aligned} & \\
\binom{N}{h\left(S\right)-1} \left(1-\epsilon\right)^{N-\left(h\left(S\right)-1\right)}. \phantom{=}&
\end{align*}

The last inequality is true because, for the first type of factors probability is less than or equal 1, and for the second type (in the product) each of the factors in the product has  probability no more than $ 1-\epsilon$. Therefore we wish to  choose $N$ in such a way that we get that $\binom{N}{(h(S)-1)} (1-\epsilon)^{N-(h(S)-1)} \leq \delta$. 

Finally, from  Lemma \ref{larga}  one can derive the bound stated in the theorem by two simple observations: 
First simply set $h=(h(S)-1)$, second the last term can be dropped because it is not positive (this is the case since $n! \geq (n/e)^n$).  
In addition one can take $r$ to be any value between $0$ and $1$. Thus, taking $r=1/2$ one gets the statement of the theorem.
\end{proof}


\begin{proof}[ Proof of Theorem \ref{quality}]
The first claim is trivial:
Since $N$ is sufficiently large, it follows from Theorem \ref{gen-calafiorecampi} that with probability at 
least $1-\delta$, the optimal solution $x^N$ of $SCP(N)$ is a feasible solution of $CCP(\epsilon)$.
Hence $J^\epsilon\leq c^T x^N=J^N$. with probability at least $1-\delta$.

For the second claim, there are two cases: $CCP(\epsilon_1)$ feasible and $CCP(\epsilon_1)$ infeasible.
If $CCP(\epsilon_1)$ is infeasible, then $J^N\leq\infty=J^{\epsilon_1}$. Suppose $CCP(\epsilon_1)$ is feasible and
consider an arbitrary $x\in K\cap S$ which is feasible for $CCP(\epsilon_1)$.
That is, let $x\in K\cap S$ such that $Pr[f(x,w)\leq0]\geq 1-\epsilon_1=(1-\delta)^{1/N}$.
Since the $N$ samples in $SCP(N)$ are independently chosen, 
the probability that $x$ is feasible for $SCP(N)$ is at least $\left((1-\delta)^{1/N}\right)^N=1-\delta$.


Since $CCP(\epsilon_1)$ is feasible, there is a sequence of vectors $(x_i)$ which are feasible for $CCP(\epsilon_1)$ such that $c^T x_i$ converges to $J^{\epsilon_1}$.
Because these vectors are feasible for $SCP(N)$ with probability at least $1-\delta$, the probability that $J^N\leq c^T x_i$ is at least $1-\delta$ for any $i\in\N$. It follows that $J^N\leq J^{\epsilon_1}$ with probability at least $1-\delta$.
\end{proof}


To conclude it is important to mention that Luedtke and Ahmed \cite{luedtke2008sample} have also studied chance constrained optimization. 
Their results are related to ours in that they obtain bounds on $N$ such that with high probability, the solution to a sampled problem is feasible in $CCP(\epsilon)$ with high probability. However, their constraints on $K$ and $f$ are different.
When $K\cap S$ is finite (e.g., purely integer variables), \cite{luedtke2008sample} showed that it is sufficient to take
$$N\geq\frac1\epsilon\ln\left(\frac1\delta\right)+\frac1\epsilon\ln\left(|K\cap S|\right).$$
This bound is better than our Theorem \ref{gen-calafiorecampi} as long as $h(S)>\ln(|K\cap S|)$ (e.g., integer variables $S=\Z^d$ since $h(S)=2^d$).

Luedtke and Ahmed also studied the mixed integer case too; for Lipschitz continuous $f$, Luedtke and Ahmed showed that if
$$N\geq\frac2\epsilon\ln\left(\frac1\delta\right)+\frac{2n}\epsilon\left\lceil\frac{2LD}\gamma\right\rceil+\frac2\epsilon\ln\left\lceil\frac2\epsilon\right\rceil,$$
where $L$ is the Lipschitz constant of $f$ and $D$ bounds the diameter of $K\cap S$, then if $f(x,w^i)\leq-\gamma$ for all $i\in[N]$, $x$ is feasible in $CCP(\epsilon)$ with probability at least $1-\delta$ (this is similar to $SCP(N)$ but with an extra tolerance of $\gamma$).
This result is similar to Luedtke and Ahmed's other result in that it depends only on the size of $K\cap S$ and not on the structure of $S$.
However, due to the tolerance constant $\gamma$, it does not say anything about the feasibility of points on the boundary of $\{f(x,w^i)\leq 0\}$ (such as the optimum).

\section{A Clarkson-type sampling algorithm for \texorpdfstring{$S$}{S}-convex optimization}

To show that our introduction of $S$-optimization makes a lot of sense we present another application besides Theorem \ref{gen-calafiorecampi}. 
We show that a Clarkson-type algorithm can be used to compute the optimal solutions to a given $S$-convex optimization problem. 
This is efficient when the number of variables is constant. We consider again the solution of $S$-optimization problem with linear objective function and convex constraints

\begin{align*}
    \min \quad & c^T \mathbf x \\
    \text{subject to} \quad & f_i(x)  \leq 0, \quad f_i \text{ convex for all } i=1,2,\ldots, m, \\
               &{\mathbf x} \in S.
\end{align*}

We demonstrate that a well-known algorithm due to Clarkson can be extended to $S$-optimization as long as $S$ is closed, has a finite Helly number $h(S)$, and one
has can have an oracle to solve deterministic small-size subproblems.
The method devised by Clarkson \cite{clarkson}  works particularly well for geometric optimization problems in few variables.  Examples of applications 
include convex and linear programming, integer linear programming, the problem of computing the minimum-volume ball or ellipsoid enclosing a given 
point set in $\R^n$, the problem of finding the distance of two convex polytopes in $\R^n$, and many others.
E.g., Clarkson stated the following result about linear programs and integer linear programs (ILPs), which gives:

\begin{theorem}[Clarkson]
Given an $m \times n$ matrix $A$, a vector $\mathbf b \in \R^m$ and the integer program $\min\{c^T \mathbf x: A\mathbf x\leq\mathbf b, \mathbf x\in \Z^n, \mathbf 0\leq \mathbf x \leq\mathbf u\}$,
one can find a solution to this problem in a expected number of steps of order $O(n^2m \log(m)) + n \log(m) O(n^{n/2})$. 
While the algorithm is exponential it gives the best complexity for solving ILPs when the number of variables $n$ is fixed.
\end{theorem}

Clarkson's algorithm requires that many small-size subsystems of the original problem are solved. This requires the call to an oracle to
solve the small systems.  The oracle originally provided by Clarkson in the case of regular integer programming was Lenstra's IP algorithm 
in fixed dimension.  As a consequence, when the number of variables is constant, Clarkson's algorithm gives a remarkable linear bound on
the complexity (see recent work by Eisenbrand \cite{eisenbrandfixeddim}). Here we prove Theorem \ref{ourclarkson} which is a direct
 generalization of Clarkson's theorem for convex continuous and integral optimization.

%

The key idea is to use the theory of \emph{violator spaces} introduced by  G\"artner, Matou{\v{s}}ek, R\"ust and \v{S}kovro\v{n} \cite{ViolatorSpaces2008}.
They showed it can be used as a general framework to work with convex optimization problems. 
Essentially, a violator space is an abstract optimization problem in which we have a finite set of constraints or elements $H$ and a function that,  given any subset 
of constraints $G$, indicates which other constraints in $H \setminus G$ \emph{violate} the feasible solutions to $G$. If one has a violator space structure,  
the optimal solution of the problem can be computed via a randomized method whose  running time is {\em linear} in the number of constraints defining the problem, 
and {\em subexponential} in the dimension  of the problem.  Violator spaces include all 
prior abstractions such as LP-type problems \cite{amenta,sharirwelzl}. The key definition from   \cite{ViolatorSpaces2008} is the following:

\begin{definition}
A {\it violator space} is a pair $(H,\vi)$, where $H$  is a finite set and $\vi$  a mapping $2^H\to2^H$, such that the following two axioms hold: 
\\\begin{tabular}{ll}
{\it Consistency}: & $G\cap \vi(G)=\emptyset$ holds for all $G\subseteq H$, and\\
{\it Locality}: & $\vi(G)=\vi(F)$ holds for all $F\subseteq G\subseteq H$ such that
$G\cap \vi(F)=\emptyset$.\\
\end{tabular}
\end{definition}

There are three important ingredients of every violator space: a basis, the combinatorial dimension, and a primitive test (which will be answered by an oracle).
First, as in the simplex method for linear programming the problem will be defined by  bases, thus we need to have a notion of basis for our optimal solutions.

\begin{definition}[{G\"artner et al. \cite{ViolatorSpaces2008}}]
A \emph{basis} of a violator space is defined in analogy to a basis of a linear programming problem: a minimal set of constraints that 
defines a solution space. Specifically, \cite[Definition~7]{ViolatorSpaces2008} defines $B\subseteq H$  
to be a  {\it basis} if  $B\cap \vi(F)\neq\emptyset$ holds for all proper subsets $F\subset B$. 
For $G\subseteq H$, a {basis of $G$}  is a minimal subset $B$ of $G$ with $\vi(B)=\vi(G)$. 

Moreover, violator space bases come with a natural combinatorial invariant, which is strongly related to the Helly numbers we discussed earlier.
The size of a largest basis of a violator space  $(H, V)$ is called the {\it combinatorial dimension}  of the violator space and denoted by $\delta=\delta(H, V)$.
\end{definition}

The primitive test operation is used as black box in all stages of the algorithm, 
is the so-called \emph{violation tests primitive}.   Given a violator space $(H,\vi)$,  some set $G \subset H$,  and some element $h\in H\setminus G$,  the  \emph{primitive} test decides whether $h\in\vi(G)$. 

G\"artner at al \cite{ViolatorSpaces2008} proved a crucial property: knowing the violations $\vi(G)$ for all $G \subseteq H$  is enough to compute a largest basis.  To do so, one can  utilize Clarkson's randomized algorithm  to compute a basis of some violator space $(H,\vi)$ with $m=\abs{H}$.

The main idea to improve over a brute-force search is due to Clarkson~\cite{clarkson}.

As described above, all one needs is to be able  to answer the  {\it Primitive query:} Given $G \subset H$ and $h \in H \setminus G$, decide whether $h \in V(G)$.  Second, the runtime is given in terms of the combinatorial dimension  $\delta(H, V)$ and the size of the input set of constraints $H$.  The key  result we will use in the rest of the paper is about the complexity of finding a basis:

\begin{theorem}\cite[Theorem~27]{ViolatorSpaces2008} \label{keytoolvio}
Using Clarkson's algorithms, a basis of $H$ in a violator space $(H,\vi)$ can be found by answering the  primitive query an expected $O\left(\delta \abs{H} + \delta^{O(\delta)}\right)$  times. 
\end{theorem}


\begin{proof}[Proof of Theorem \ref{ourclarkson}]
Let $H =\{ f_1, f_2, \dots, f_m\}$ be the constraints of the $S$-convex optimization problem of the statement of Theorem \ref{ourclarkson}. 
We define a the violator set operator $\vi(G)$ for a subset of inequalities $G \subset H$ as follows:
We provide each $S$-program with a universal tie-breaking rule, for instance, using lexicographic ordering.
A constraint $h \in H$ is in $\vi(G)$ if the  optimal solution value of the subsystem $G$ 
with respect to the objective function, denoted $\vec{x}_G$,  
is not  equal to the unique optimal solution of $G \cup\{h\}$, denoted $\vec{x}_{G \cup \{h\}}$.
Note that we need to have a total ordering on the possible feasible
solutions of $G$ and the fact that $S$ is closed to have a unique optimum. 


For our proof we define the violator map as follows: a constraint $h \in H$ is in $\vi(G)$ if
the optimal solutions satisfy $\vec{x}_G > \vec{x}_{G \cup \{h\}}$.
If we assume that $G$ has no feasible solutions, we define $V(G)$ as being the empty set. 
Indeed any new constraint added to the integer program can only decrease the number of feasible solutions.
We need to check that the two conditions presented in the definition of violator spaces.
are satisfied. The consistency condition is clearly satisfied. 

Assume now that $F \subseteq G \subseteq H$ and $G \cap \vi(F) = \emptyset$.
To show locality we must verify that $\vi(F) = \vi(G)$.
Note that by the hypothesis  $G \cap \vi(F)$ it means that $\vec{x}_G=\vec{x}_{F}$ because
otherwise at least one element in $G$ must be in $\vi(F)$.

Now we verify first the containment $\vi(F) \subseteq \vi(G)$. Take $h\in \vi(F)$; if $h \notin \vi(G)$ then 
$\vec{x}_{G \cup \{h\}} =\vec{x}_G=\vec{x}_{F} > \vec{x}_{F \cup \{h\}}$. However, $F \cup \{h\} \subset G \cup \{h\}$.
It follows that $\vec{x}_{F \cup \{h\}} \geq \vec{x}_{G \cup \{h\}}$ too---a contradiction.
Now we check $\vi(G) \subset \vi(F)$. Take $h \in \vi(G)$, if  $\vec{x}_{F \cup \{h\}}=\vec{x}_{F} = \vec{x}_G >  \vec{x}_{G \cup \{h\}} $
But then there exist $g \in G$ such that $g \in \vi(F \cup \{h\})=\vi(F)$ a contradiction.

Since the two conditions of a violator space are satisfied, all that is left to apply Theorem \ref{keytoolvio} is to outline what the combinatorial dimension and
the primitive test are. First, a basis for $G$, using this violator space,  represents an optimal solution of the $S$-subproblem. 
But if  we have the optimum value $\vec{x}_G$, then the optimal solution is defined by no more than $(h(S)-1)$ of the $f_i$. 
This is because $f_i, i=1\dots N$ together with $c^Tx<c^Tx_N$ is an  $S$-convex set which has no solutions in $S$. Thus by the definition of the $S$-Helly number, there are no
more than $h(S)$ infeasible subfamilies, but this means that from the original $f_i \in G$ only $h(S)-1$ participate. Therefore the combinatorial dimension of this violator space is $h(S)-1$.
The primitive test is provided by an oracle that solves smaller problems of size $O(h(S))$. Therefore, the conclusion of Theorem \ref{ourclarkson} follows by applying Theorem \ref{keytoolvio}.
\end{proof}

\section{Concluding Remarks} We have shown that the quality guarantees of the sampling method of Calafiore and Campi can be extended to more abstract convex optimization constraints. Clearly the value of these results depends on having a practical algorithm to solve $SCP(N)$. Similarly, Clarkson's method the query $h\in \vi(C)$ is answered  via calls to the primitive as a black box or oracle.  The algorithms we derive are randomized but run in expected polynomial time complexity when the number of
discrete variables is fixed. Moreover the algorithmic complexity is in fact linear in the number of constraints, and it depends on calls to an oracle that solves small size subproblems. The size of these smaller subproblems is precisely the $S$-Helly number. 

In both cases, one requires an oracle to solve or test feasibility of a small-size $S$-convex algorithms. These exist for $S=\R^d$, and for $S=\Z^d, S=\Z^{d-k}\times\R^k$, as presented by the usual deterministic algorithms for mixed-integer convex optimization.  It is possible to prove, using the results of \cite{barvipom}, that for $S$ equal to the difference of a lattice with the union of finitely many of its sublattices, one can have such an algorithm when all the constraints define  a polyhedron of fixed dimension. In a forthcoming paper we will present experiments that use the sampling bounds shown here to solve some of these problems. The development of other such oracles will require the development of some interesting mathematics.

\section*{Acknowledgements} The authors are grateful to Shabbir Ahmed for introducing them to this subject. 
This research was supported by a UC MEXUS grant that helped established the collaboration of the UC Davis and UNAM teams. We are grateful for the support.
The first, second and third author travel was supported in part by the Institute for Mathematics  and its 
Applications and an NSA grant. The third and fourth authors  were also supported by CONACYT project 166306.

\bibliographystyle{amsplain}
\bibliography{chanceconstrained}
\end{document}